\documentclass[12pt]{amsart}
\usepackage[utf8]{inputenc}
\usepackage[foot]{amsaddr}
\usepackage{amsmath,amsfonts,amsthm,amssymb,verbatim,etoolbox,color}
\usepackage{enumitem} 
\usepackage{epigraph}
\usepackage{appendix}
\usepackage{xfrac}
\usepackage{mathtools}
\usepackage{flexisym}
\usepackage{bbm}

\usepackage[margin=1in]{geometry}
\usepackage[
    backend=biber,
    maxnames=4,
    maxalphanames=4,
    style=alphabetic,
    backref
  ]{biblatex}
\usepackage{scalerel}

\usepackage{hyperref}
\hypersetup{
    colorlinks=true,
    linkcolor=red,
    filecolor=magenta,      
    urlcolor=magenta,
    pdftitle={},
    pdfauthor = {Fred Tyrrell},
    pdfpagemode=FullScreen,
    linktocpage = true,
    citecolor = blue,
    }

\usepackage{theoremref}

\newtheorem{thm}{Theorem}
\newtheorem{lemma}[thm]{Lemma}

\newtheorem{corollary}[thm]{Corollary}
\newtheorem{conjecture}[thm]{Conjecture}

\theoremstyle{definition}

\newtheorem{remark}[thm]{Remark}

\numberwithin{thm}{section}

\newcommand{\F}{\mathbb{F}}

\newcommand{\sub}{\subseteq}

\newcommand{\bra}[1]{\left(#1\right)}

\newcommand{\ra}{\rightarrow}

\newcommand{\abs}[1]{\left\lvert {#1} \right \rvert}

\title{Beating product constructions for linear equations over finite fields}
\author{Paul Hametner}
\email{paul.hametner@ista.ac.at}
\address{Institute of Science and Technology, Austria}
\author{Fred Tyrrell}
\email{fred.tyrrell@bristol.ac.uk}
\address{Fry Building, School of Mathematics, University of Bristol}
\date{\today}

\urladdr{http://fredtyrrell.com}

\begin{document}
\begin{abstract}
We show that for any $A\subseteq \F_q^n$ lacking non-trivial solutions to a translation-invariant linear equation of genus one, meaning that no nonempty proper subset of the coefficients sums to $0$, there is a set $B\subseteq \F_q^m$ in some higher dimension which also lacks non-trivial solutions, such that
\[|B|^{1/m}>|A|^{1/n}.\]
In particular, this implies that no fixed cap set in $\F_3^n$ gives an asymptotically optimal lower bound by direct products alone.
\end{abstract}
\maketitle

\section{Introduction}
\subsection{Cap sets}
A cap set is a set $A \sub \F_3^n$ with no solutions to $x+y+z = 0$ with $x,y,z \in A$, other than the trivial solutions $x=y=z$. Since $x+y+z = 0 \iff x+z=2y$ in $\F_3^n$, this is equivalent to the property that $A$ contains no non-trivial arithmetic progressions. We write
\[r_3\bra{\F_3^n}:= \max \{\abs{A} : A \sub \F_3^n\text{ is a cap set}\},\]
and the cap set problem asks for the asymptotic behaviour of $r_3\bra{\F_3^n}$ as $n \ra \infty$.
\smallbreak
A breakthrough result of Ellenberg and Gijswijt \cite{EllenbergGijswijt}, based on the polynomial method of Croot, Lev and Pach \cite{CLP}, shows that $r_3\bra{\F_3^n} \leq \bra{2.7552\ldots}^n$, thus proving an exponential upper bound for the size of a maximal cap set.
\smallbreak
On the other hand, obtaining an exponential lower bound on $r_3\bra{\F_3^n}$ is rather more straightforward. If $A\subseteq \F_3^n$ and $B\subseteq \F_3^m$ are cap sets, then it is easy to show (see Proposition 2.1 in \cite{tyrrell2023new}) that $A\times B\subseteq \F_3^{n+m}$ is again a cap set. Thus by taking direct products of $A$ with itself, one can show (Proposition 2.2 in \cite{tyrrell2023new}) that any cap set $A\subseteq \F_3^n$ of size $c^n$ gives the asymptotic lower bound
\[r_3(\F_3^m)\geq \bra{c-o(1)}^m.\]
\smallbreak
The best known lower bounds for $r_3(\F_3^n)$ have ultimately come from finite-dimensional constructions, which are then turned into asymptotic lower bounds by product-type arguments. Edel \cite{Edel} introduced an extended product construction, building on work of Calderbank and Fishburn \cite{CalderbankFishburn}, which was later developed further by the second author \cite{tyrrell2023new}. AI-assisted searches have since led to further improvements, first through FunSearch \cite{romera2024mathematical} and more recently through $X$-evolve \cite{zhai2025x}, which find improved admissible sets to use in the framework of \cite{tyrrell2023new} and \cite{Edel}. The current record is $r_3\bra{\F_3^n} \geq \bra{2.2203\ldots}^n$.
\smallbreak
It is therefore natural to ask whether direct products of a given cap set could ever be asymptotically optimal. This was raised explicitly in the editorial introduction to \cite{tyrrell2023new}, and was asked by the second author during the problems session at the 30th British Combinatorial Conference (\cite{cameron2024problems}, Problem 21), and earlier by Sean Eberhard\footnote{Communicated to us by Thomas Bloom.}. The first result of this paper answers this question in the affirmative.

\begin{thm}\thlabel{capmain}
        Let $A \sub \F_3^n$ be a cap set. Then, for every integer $r \geq 1$, there is a cap set $B \sub \F_3^{2nr}$ of size \[\abs{B}=r\abs{A}^{2r-2}.\]
\end{thm}
Taking $r > \abs{A}^2$ in \thref{capmain}, we have the following.
\begin{corollary}\thlabel{capcor}
    Let $A \sub \F_3^n$ be a cap set of size $c^n$. Then there exists a cap set $B \sub \F_3^m$ for some $m>n$, such that $\abs{B}>c^m$.
\end{corollary}

\subsection{Genus one equations}
A similar result applies to more general linear equations over finite fields. Let $q$ be a prime power and denote by $\F_q$ the finite field with $q$ elements. Let $s\geq 3$ and consider the linear form
\[L\bra{x_1,\ldots,x_s}:=c_1x_1+\cdots+c_sx_s\]
over $\F_q$ with nonzero coefficients $c_i \in \F_q \setminus \{0\}$. We say that $L$ is translation-invariant if
\[c_1+\cdots+c_s=0.\]
Ruzsa \cite{ruzsa1993solving} introduced the concept of the genus of an equation over the integers, and the same language is used in the finite-field setting by Mimura and Tokushige \cite{mimura2021solving}. A translation-invariant $L$ has \emph{genus one} if
\[\sum_{i\in I}c_i\neq 0\]
for every nonempty proper subset $I\subsetneq [s]$. When $L$ is translation-invariant and of genus one, we say that a set $A\sub \F_q^n$ is \emph{$L$-free} if there are no solutions to $L\bra{x_1,\ldots,x_s}=0$ with $x_i\in A$, except for the trivial solutions where $x_1=\cdots=x_s$.

\begin{thm}\thlabel{genusmain}
Let $L\bra{x_1,\ldots,x_s}=0$ be a genus one translation-invariant equation over $\F_q$, where $s \geq 3$. Let $d=\bra{2^s-s-2}\bra{1+2s\bra{s-1}}$. If $A\sub \F_q^n$ is $L$-free, then, for every integer $r\geq 1$, there is an $L$-free set $B\sub \F_q^{dnr}$ of size
\[\abs{B}=r\abs{A}^{d\bra{r-1}}.\]
\end{thm}
As with cap sets, direct products of $L$-free sets are $L$-free. Thus \thref{genusmain}, applied with any $r>\abs{A}^{d}$, gives a strict improvement over the product lower bound.
\begin{corollary}\thlabel{genuscor}
Let $L\bra{x_1,\ldots,x_s}=0$ be a genus one translation-invariant equation over $\F_q$ with $s \geq 3$, and let $A\sub \F_q^n$ be $L$-free of size $c^n$. Then there exists an $L$-free set $B\sub \F_q^m$ for some $m>n$ such that $\abs{B}>c^m$.
\end{corollary}

\begin{remark}
    The requirement that $L$ be genus one is completely natural in this context, as one cannot obtain asymptotic lower bounds for higher genus equations via direct products. For higher genus equations, the natural notion of a trivial solution is no longer simply that all variables are equal, and therefore direct products do not preserve $L$-freeness. 
    \smallbreak
    For example, consider the Sidon equation $L\bra{x_1,x_2,x_3,x_4}=x_1+x_2-x_3-x_4 = 0$, which is not genus one. One calls a solution trivial for this equation if $\{x_1,x_2\}=\{x_3,x_4\}$, and it follows that the direct product of Sidon sets is not necessarily Sidon.
\end{remark}
\begin{remark}
    A similar non-optimality result appears in work of Alon and Kleitman \cite{alon1990sum}, in the context of sum-free sets of integers. If $s(A)$ denotes the largest subset of $A$ which is sum-free, meaning no solutions to $x+y = z$, Corollary 2.3 of \cite{alon1990sum} says that there is always a $B$ such that $\frac{s(B)}{\abs{B}}<\frac{s(A)}{\abs{A}}$. Thus, there is no single set $A$ which is extremal for the ratio $\frac{s(A)}{\abs{A}}$.
\end{remark}

\subsection{Overview of the paper}

To prove \thref{capmain,genusmain}, we carefully construct a collection of direct products from linear transformations of $A$, and then show that the union of these direct products does not contain any non-trivial solutions. This is similar to the extended product construction introduced by Edel \cite{Edel} and developed by the second author \cite{tyrrell2023new}.
\smallbreak
The proof of \thref{capmain} is a shorter and more efficient version of the proof of \thref{genusmain}, but uses the same overall strategy. We therefore prove \thref{capmain} first in Section 2, and then \thref{genusmain} in Section 3.

\section*{Acknowledgements}
We thank Thomas Bloom for useful discussions and encouragement, particularly for suggesting that the original construction for cap sets in \thref{capmain} could be generalised to other translation-invariant equations. We also thank Christian Elsholtz for comments on an earlier version of this paper. The authors would like to thank the Erd\H{o}s Centre and the R\'{e}nyi Institute for accommodating them together for the duration of two summer schools and two workshops, which led to them collaborating on this paper. The second author is supported by the Heilbronn Doctoral Partnership.

\section{There is no asymptotically optimal cap set}
Throughout this section, we use the fact that cap sets are preserved under invertible affine transformations, and that direct products of cap sets are cap sets.
\smallbreak
To prove \thref{capmain}, we require the following lemma.
\begin{lemma}\thlabel{squaredisjoint}
    Let $A \sub \F_3^n$ be a cap set. Then there are two cap sets $A_1, A_2 \sub \F_3^{2n}$ such that
    \begin{enumerate}
        \item $\abs{A_1}=\abs{A_2}=\abs{A}^2$,
        \item $0 \notin A_1 + A_2$,
        \item $0 \notin A_1 \cup A_2$.
    \end{enumerate}
\end{lemma}
\begin{proof}
By translating $A$ if necessary, we may assume without loss of generality that $0 \notin A$. 
\smallbreak
Let $A_1 = A \times A$. Then $A_1 \sub \F_3^{2n}$ is a cap set of size $\abs{A}^2$.
\smallbreak
Let $T\bra{x,y}=\bra{x+y,x-y}$, which is an invertible linear map over $\F_3$. Then 
\[A_2 = T\bra{A_1} = \{\bra{x+y,x-y}:x,y \in A\}\] 
is a cap set with $\abs{A_2}=\abs{A_1}=\abs{A}^2$.
\smallbreak
Assume for a contradiction that $0\in A_1+A_2$. Then there exist $a,b,x,y\in A$ such that
\[(a,b)+(x+y,x-y)=(0,0).\]
The first coordinate gives $a+x+y=0$. Since $A$ is a cap set, this implies $a=x=y$. The second coordinate then gives $b+x-y=0$, and hence $b=0$, contradicting $0\notin A$. Therefore $0\notin A_1+A_2$.
\smallbreak
Finally, since $0 \notin A$, we certainly have $0 \notin A \times A = A_1$, and since $T$ is an invertible linear map, we also have $0\notin T\bra{A_1}$.
\end{proof}

We are now ready to prove \thref{capmain}.
\begin{proof}[Proof of \thref{capmain}]
Starting with a cap set $A \sub \F_3^n$, we apply \thref{squaredisjoint} to obtain cap sets $A_1,A_2 \subseteq \F_3^{2n}$ with
\[|A_1|=|A_2|=\abs{A}^2,\qquad 0 \notin A_1 \cup A_2,\qquad 0 \notin A_1+A_2.\]
Fix an integer $r \geq 1$, and define, for $1 \leq j \leq r$, the direct products
\[B_j = A_1^{j-1} \times \{0\}^{2n} \times A_2^{r-j}.\]
Since $A_1,A_2$ are cap sets, and trivially $\{0\}$ is a cap set, each $B_j$ is a cap set in $\F_3^{2nr}$ of size
\[\abs{A_1}^{j-1}\abs{A_2}^{r-j}=\abs{A}^{2r-2}.\]
Let
\[B = \bigcup_{j=1}^r B_j.\]
For $b \in B$ and $t \in [r]$ we write $b^{\bra{t}}\in\F_3^{2n}$ for the $t$-th block of $b$, so \[b^{\bra{t}}=\bra{b_{2n(t-1)+1},\ldots,b_{2nt}} \in \F_3^{2n}.\]
\smallbreak
If $u \in B_i$, $v \in B_j$, where $i \neq j$, then $u^{(i)} =0$, and $v^{(i)} \in A_1 \cup A_2$. Therefore, since $0 \notin A_1\cup A_2$, we have $B_i \cap B_j = \emptyset$ if $i \neq j$, and hence
\[\abs{B}=r \abs{A}^{2r-2}.\] 
\smallbreak
To show that $B$ is a cap set, assume we have distinct $x,y,z \in B$ such that $x+y+z=0$. Since each $B_i$ is a cap set, we cannot have $x,y,z \in B_i$. Thus, we either have (permuting $x,y,z$ if necessary) $x,y \in B_i$, $z \in B_j$ or $x \in B_i$, $y \in B_j$, $z \in B_k$, where $i,j,k \in [r]$ are distinct. 
\smallbreak
Since $x+y+z=0$, we have
\[x^{(t)}+y^{(t)}+z^{(t)}=0 \qquad \bra{1 \leq t \leq r}.\]
\smallbreak
In the first case, where $x,y \in B_i$, $z \in B_j$ for $i \neq j$, we have $x^{\bra{i}} = y^{\bra{i}}=0$. Depending on whether $i<j$ or $i > j$, we have $z^{\bra{i}} \in A_1$ or $z^{\bra{i}} \in A_2$ respectively. But $x^{\bra{i}} + y^{\bra{i}} + z^{\bra{i}} = 0$, and hence $z^{(i)}=0$, which is a contradiction, since $0 \notin A_1 \cup A_2$.
\smallbreak
In the second case, where $x,y,z$ are in $B_i,B_j,B_k$ respectively, assume without loss of generality that $i<j<k$. Then we have
\[    x^{(j)} \in A_2, \qquad y^{(j)} =0,\qquad    z^{(j)} \in A_1.\]
Thus $x^{(j)}+z^{(j)}=0$, but this is a contradiction, since $0 \notin A_1+A_2$.
\smallbreak
Thus there are no distinct $x,y,z \in B$ such that $x+y+z=0$, so $B$ is indeed a cap set, which proves the result.
\end{proof}
\begin{remark}
    In the language of \cite{tyrrell2023new}, this is the extended product construction applied to the extendable collection \(\{0\}^{2n},A_1,A_2\) from \thref{squaredisjoint} and the admissible set \( I(r,r-1)\) from Lemma 2.12 of \cite{tyrrell2023new}. See Section 2 of \cite{tyrrell2023new} for the details of the extendable-admissible construction.
\end{remark}

\section{Genus one equations}
We now prove \thref{genusmain}. Throughout this section, let
\[L\bra{x_1,\ldots,x_s}=c_1x_1+\cdots+c_sx_s\]
be a genus one translation-invariant linear form over $\F_q$, with $s \geq 3$. We shall use without further comment that, since \(L\) is translation-invariant, \(L\)-freeness is preserved under translations and invertible linear transformations. We shall also use that direct products of \(L\)-free sets are \(L\)-free, since $L$ has genus one.
\smallbreak
The proof of \thref{genusmain} has two ingredients. First, after replacing the original set by a direct product of translates of itself, we may assume that all partial equations involving at most $s-2$ variables have no solutions. Second, we use a two-coordinate construction to handle the remaining case, where all variables lie in different zero-block positions.

\begin{lemma}\thlabel{partialavoidance}
Let $A\sub \F_q^n$ be $L$-free, and let $k=2^s-s-2$. Then there is an $L$-free set $A'\sub \F_q^{kn}$ with $\abs{A'}=\abs{A}^k$ such that, for every nonempty $I\sub [s]$ with $\abs{I}\leq s-2$, there are no solutions to
\[\sum_{i\in I}c_ix_i=0,\qquad x_i\in A'.\]
In particular, $0\notin A'$.
\end{lemma}
\begin{proof}
If $\abs{A} = 1$ the lemma directly follows from the genus one assumption, so we assume $\abs{A} \geq 2$. Fix a nonempty set $I\sub [s]$ with $\abs{I}\leq s-2$. We first show that $\sum_{i \in I}c_iA \neq \F_q^n$. 
\smallbreak
Choose distinct $a,b\in A$. Since $\abs{[s]\setminus I}\geq 2$, we may choose $j\in [s]\setminus I$ and set $x_{j}=b$ and $x_i=a$ for all $i\in [s]\setminus(I\cup\{j\})$. If $\sum_{i \in I}c_iA=\F_q^n$, then we can choose $x_i\in A$ for $i\in I$ such that
\[\sum_{i\in I}c_ix_i=-\sum_{i\in [s]\setminus I}c_ix_i.\]
This gives a non-diagonal solution to $L(x_1,\ldots,x_s)=0$ in $A$, contradicting the assumption that $A$ is $L$-free. Hence $\sum_{i \in I}c_iA\neq \F_q^n$.
\smallbreak
Since $\sum_{i \in I}c_i A$ is a proper subset of $\F_q^n$, we can choose $t_I\in \F_q^n$ such that
\[0\notin \sum_{i \in I}c_i\bra{A+t_I},\]
since $\sum_{i \in I}c_i \neq 0$ as $L$ has genus one. 
\smallbreak
Now let
\[A'=\prod_{\substack{I\sub [s]\\1\leq \abs{I}\leq s-2}}(A+t_I).\]
Then $A'$ is $L$-free, and $\abs{A'}=\abs{A}^k$, where $k=2^s-s-2$ is the number of nonempty subsets $I\sub [s]$ with $\abs{I}\leq s-2$. Finally, if there were a solution to $\sum_{i\in I}c_ix_i=0$ in $A'$ for some such $I$, then projecting to the coordinate indexed by $I$ would contradict the choice of $t_I$.
\end{proof}

The next lemma is the local two-coordinate construction which will be used when all variables lie in different zero-block positions.
\begin{lemma}\thlabel{twocoordinate}
Let $A\sub \F_q^n$ be $L$-free, and suppose that there are no solutions to
\[\sum_{i\in I}c_ix_i=0,\qquad x_i\in A,\]
for every nonempty $I\sub [s]$ with $\abs{I}\leq s-2$. Let $\ell,m\in [s]$ be distinct, and define
\[T_{\ell,m}\bra{x,y}=\left(x+\frac{c_{m}}{c_{\ell}}y,x-y\right).\]
Then $A\times A$ and $T_{\ell,m}(A\times A)$ are $L$-free sets of size $\abs{A}^2$, and there are no solutions to
\[c_{\ell}y+\sum_{i\in [s]\setminus\{\ell,m\}}c_ix_i=0\]
with $y\in T_{\ell,m}(A\times A)$ and $x_i\in A\times A$.
\end{lemma}
\begin{proof}
By the genus one assumption, $c_m \neq -c_l$, so that $T_{\ell,m}$ is invertible, and the sets $A\times A$ and $T_{\ell,m}(A\times A)$ are $L$-free sets of size $\abs{A}^2$.
\smallbreak
Suppose for a contradiction that there are $y \in T_{\ell,m}\bra{A \times A}$ and $x_i \in A\times A$ for $i \in [s]\setminus\{\ell,m\}$ such that
\[c_{\ell}y+\sum_{i \in [s]\setminus\{\ell,m\}}c_ix_i =0.\] 
Then there are $a,b \in A$ such that
\[y=\left(a+\frac{c_{m}}{c_{\ell}}b,a-b\right),
\]
and for each $i\in [s]\setminus\{\ell,m\}$, we have
\[x_i=\bra{a_i,b_i},\qquad a_i,b_i\in A.\]
The first coordinate gives
\[c_{\ell}a+c_{m}b+\sum_{i\in [s]\setminus\{\ell,m\}}c_ia_i=0.\]
This is a solution to $L=0$ in $A$, so it is diagonal. In particular, $a=b$. The second coordinate then gives
\[\sum_{i\in [s]\setminus\{\ell,m\}}c_ib_i=0,\]
which contradicts the assumed absence of partial solutions, since $[s]\setminus\{\ell,m\}$ is nonempty and has size $s-2$.
\end{proof}

We can now prove \thref{genusmain}, by combining \thref{partialavoidance} and \thref{twocoordinate}, using a similar construction to the proof of \thref{capmain}.
\begin{proof}[Proof of \thref{genusmain}]
Let $A\sub \F_q^n$ be $L$-free. Apply \thref{partialavoidance} to obtain an $L$-free set $A'\sub \F_q^{kn}$ with $\abs{A'}=\abs{A}^k$ such that every nonempty partial equation of size at most $s-2$ has no solutions in $A'$. By the last assertion of \thref{partialavoidance}, we also have $0\notin A'$.
\smallbreak
Now define 
\[A_1=A'\times\prod_{\substack{(\ell,m)\in[s]^2\\\ell \neq m}}(A'\times A'),\qquad A_2=A'\times\prod_{\substack{(\ell,m)\in[s]^2\\\ell \neq m}}T_{\ell,m}(A'\times A').\]
Put
\[d=k\bra{1+2s(s-1)}.\]
Then
\[A_1,A_2\sub \F_q^{dn},\qquad \abs{A_1}=\abs{A_2}=\abs{A}^d.\]
Moreover, $A_1$ and $A_2$ are $L$-free, and $0\notin A_1\cup A_2$, since $0\notin A'$.
\smallbreak
For $1\leq j\leq r$, define
\[B_j=A_1^{j-1}\times\{0\}^{dn}\times A_2^{r-j}\sub \F_q^{dnr}, \qquad B=\bigcup_{j=1}^rB_j.\]
The sets $B_j$ are pairwise disjoint, since their zero blocks occur in different positions and $0\notin A_1\cup A_2$. Therefore
\[\abs{B}=r\abs{A_1}^{r-1}=r\abs{A}^{d\bra{r-1}}.\]
\smallbreak
It remains to show that $B$ is $L$-free. Assume there is a solution
\[\sum_{i=1}^s c_ib_i=0,\qquad b_i\in B.\]
For each $i\in [s]$, let $\rho(i)\in [r]$ such that $b_i\in B_{\rho(i)}$. For $b\in B$ write $b^{\bra{t}}$ for the $t$-th block of $b$. Then
\[\sum_{i=1}^s c_ib_i^{\bra{t}}=0\qquad \bra{1\leq t\leq r}.\]
We split into three cases.
\smallbreak

\begin{enumerate}[label=(\roman*)]
    \item First suppose that $\rho(1) = \rho(2) = \dots = \rho(s) = j \in [r]$. Then since each $B_j$ is $L$-free, we must have $b_1=\cdots=b_s$.
    \item Next suppose that there is some $j\in [r]$ such that
\[2\leq \abs{\{i:\rho(i)=j\}}\leq s-1.\]
Let
\[I=[s]\setminus\{i:\rho(i)=j\}.\]
Looking in the $j$-th block gives
\[\sum_{i\in I}c_ib_i^{\bra{j}}=0.\]
Projecting to the first $A'$-factor of $A_1$ and $A_2$ gives a solution to
\[\sum_{i\in I}c_ix_i=0,\qquad x_i\in A'.\]
Here $1\leq \abs{I}\leq s-2$, contradicting the partial-avoidance property of $A'$.

\item It remains to consider the case where all the values $\rho(i)$ are distinct. Let $\ell\in [s]$ be such that $\rho(\ell)$ is minimal, and let $m\in [s]$ be such that $\rho(m)$ is the second smallest value among the $\rho(i)$. Looking in the $\rho(m)$-th block, the $m$-th variable vanishes, the $\ell$-th variable lies in $A_2$, and every other variable lies in $A_1$. Projecting to the auxiliary factor indexed by $(\ell,m)$ gives a solution to
\[c_{\ell}y+\sum_{i\in [s]\setminus\{\ell,m\}}c_ix_i=0\]
with $y\in T_{\ell,m}(A'\times A')$ and $x_i\in A'\times A'$. This contradicts \thref{twocoordinate}.
\end{enumerate}
Thus every solution in $B$ is diagonal, so $B$ is $L$-free.
\end{proof}

\section{Discussion}
\begin{remark}
    The improvement in \thref{capcor} is not a quantitatively useful way to achieve better lower bounds for the cap set problem. \thref{capmain} gives an improvement to the constant $c$ in \thref{capcor} by a factor of
    \[\bra{\frac{r}{\abs{A}^2}}^{\frac{1}{2nr}} \leq \exp\bra{\frac{1}{2ne\abs{A}^2}}.\]
    Indeed, a calculation shows that \thref{capmain} can improve the current record for the lower bound of $c=2.2203\ldots$ from \cite{zhai2025x} in the 452nd decimal place. Reapplying the construction gives further improvements, which get super-exponentially smaller. Thus, the main result of \thref{capmain,capcor} is that \emph{some} improvement is always possible, and not the actual value of the improvement.
\end{remark}

\begin{remark}
    The dimension parameter $d$ in \thref{genusmain} can be reduced with a slightly more technical argument, particularly if there is symmetry among the coefficients of $L$. More precisely, if there are $t$ distinct coefficients appearing in \(L\), with multiplicities $m_1,\ldots,m_t$, let
\[d_0=\prod_{j=1}^t\bra{m_j+1}-2t-2, \qquad d_1=t\bra{t-1}+\#\{j\in[t]:m_j\geq 2\}.\]
Then \thref{genusmain} holds with $d=d_0+2d_1$, which is $2^s-2s-2 + 2s\bra{s-1}$ when the coefficients of $L$ are all distinct, but can in general be much smaller. 
\smallbreak
For example, with $L\bra{x,y,z}=x+z-2y$, the equation of a 3-term-progression, \thref{genusmain} gives $d=39$, but the above gives $d=6$ in odd characteristic greater than $3$, and $d=2$, as in \thref{capmain}, in characteristic $3$, when this equation is $x+y+z=0$.
\smallbreak
However, even after this optimisation, the improvement that \thref{genuscor} gives to the constant $c$ is extremely small. Using $d=6$ in \thref{genusmain} still does not give a meaningful improvement to the lower bound for 3-term-progression-free sets over $\F_p^n$ for $p>5$ due to Elsholtz, Hunter, Proske and Sauermann \cite{elsholtz2024improving}, and over $\F_5^n$ due to Pollak \cite{Pollak2025}. Therefore, we have chosen to present the cleanest proof of the qualitative result, \thref{genuscor}, and not optimise $d$ in \thref{genusmain}.
\end{remark}

\begin{remark}
It would be natural to consider whether \thref{genuscor} also holds for systems of equations. In this context, we call a system $L=\bra{L_1,\ldots,L_t}$ \emph{translation-invariant} if each $L_j$ is a translation-invariant equation in $s$ variables. 
\smallbreak
Writing
\[L_j\bra{x_1,\ldots,x_s}=c_{1,j}x_1 + \cdots + c_{s,j}x_s,\] the natural generalisation of genus one to a system of equations is that there is no nonempty proper subset $I \subsetneq [s]$ such that
\[\sum_{i \in I}c_{i,j} = 0, \qquad \bra{1 \leq j \leq t}.\]
\smallbreak
We say $A \sub \F_q^n$ is $L$-free if the only solutions to
\[L_j\bra{x_1,\ldots,x_s}=0, \qquad \bra{1 \leq j \leq t}\]
are the diagonal solutions with $x_1 = \cdots = x_s$. 
\smallbreak
We call a system \emph{non-degenerate} if it has a solution over $\F_q$ with not all variables equal. We need to exclude degenerate systems, since if $L$ is degenerate then $\F_q^n$ is $L$-free, so clearly a result like \thref{genuscor} cannot be obtained.
\smallbreak
With these definitions, if $L$ is a translation-invariant system of genus one, $L$-freeness is preserved under invertible affine transformations and direct products. In light of this, we expect that a result like \thref{genuscor} should hold for systems of equations.
\end{remark}

\begin{conjecture}
    Let $L$ be a non-degenerate genus one translation-invariant system of equations over $\F_q$, and let $A \sub \F_q^n$ be $L$-free of size $c^n$. Then there exists an $L$-free set $B \sub \F_q^m$ for some $m>n$ such that $\abs{B}>c^m$.
\end{conjecture}

\printbibliography
\end{document}